\documentclass[12pt]{amsart}
\usepackage{mathrsfs}      
\usepackage{txfonts}
\usepackage{amssymb}
\usepackage{eucal}
\usepackage{graphicx}
\usepackage{amsmath}
\usepackage{amscd}
\usepackage[all]{xy}           
\usepackage[active]{srcltx} 

\usepackage{amsfonts,latexsym}
\usepackage{xspace}
\usepackage{epsfig}
\usepackage{float}
\usepackage{color}
\usepackage{fancybox}
\usepackage{colordvi}
\usepackage{multicol}
\usepackage{colordvi}
\usepackage[hypertex]{hyperref} 

\newtheorem{theorem}{Theorem}[section]
\newtheorem{prop}[theorem]{Proposition}
\newtheorem{defn}[theorem]{Definition}
\newtheorem{lemma}[theorem]{Lemma}
\newtheorem{coro}[theorem]{Corollary}
\newtheorem{prop-def}{Proposition-Definition}[section]
\newtheorem{coro-def}{Corollary-Definition}[section]

\newtheorem{remark}[theorem]{Remark}

\newtheorem{exam}[theorem]{Example}

\newcommand{\nc}{\newcommand}
\nc{\tred}[1]{\textcolor{red}{#1}}
\nc{\tblue}[1]{\textcolor{blue}{#1}}
\nc{\tgreen}[1]{\textcolor{green}{#1}}
\nc{\tpurple}[1]{\textcolor{purple}{#1}}
\nc{\btred}[1]{\textcolor{red}{\bf #1}}
\nc{\btblue}[1]{\textcolor{blue}{\bf #1}}
\nc{\btgreen}[1]{\textcolor{green}{\bf #1}}
\nc{\btpurple}[1]{\textcolor{purple}{\bf #1}}

\newcommand{\efootnote}[1]{}

\renewcommand{\textbf}[1]{}

\newcommand{\delete}[1]{}
\nc{\dfootnote}[1]{{}}          
\nc{\ffootnote}[1]{\dfootnote{#1}}
\nc{\mfootnote}[1]{\footnote{#1}} 
\nc{\ofootnote}[1]{\footnote{\tiny Older version: #1}}

\nc{\mlabel}[1]{\label{#1}}  
\nc{\mcite}[1]{\cite{#1}}  
\nc{\mref}[1]{\ref{#1}}  
\nc{\mbibitem}[1]{\bibitem{#1}} 

\delete{
\nc{\mlabel}[1]{\label{#1}  
{\hfill \hspace{1cm}{\bf{{\ }\hfill(#1)}}}}
\nc{\mcite}[1]{\cite{#1}{{\bf{{\ }(#1)}}}}  
\nc{\mref}[1]{\ref{#1}{{\bf{{\ }(#1)}}}}  
\nc{\mbibitem}[1]{\bibitem[\bf #1]{#1}} 
}

\nc{\opa}{\ast} \nc{\opb}{\odot} \nc{\op}{\bullet} \nc{\pa}{\frakL}
\nc{\arr}{\rightarrow} \nc{\lu}[1]{(#1)} \nc{\mult}{\mrm{mult}}
\nc{\diff}{\mathfrak{Diff}}
\nc{\opc}{\sharp}\nc{\opd}{\natural}
\nc{\ope}{\circ}
\nc{\bin}[2]{ (_{\stackrel{\scs{#1}}{\scs{#2}}})}  
\nc{\binc}[2]{ \left (\!\! \begin{array}{c} \scs{#1}\\
    \scs{#2} \end{array}\!\! \right )}  
\nc{\bincc}[2]{  \left ( {\scs{#1} \atop
    \vspace{-1cm}\scs{#2}} \right )}  
\nc{\bs}{\bar{S}} \nc{\cosum}{\sqsubset} \nc{\la}{\longrightarrow}
\nc{\rar}{\rightarrow} \nc{\dar}{\downarrow} \nc{\dprod}{**}
\nc{\dap}[1]{\downarrow \rlap{$\scriptstyle{#1}$}}
\nc{\md}{\mathrm{dth}} \nc{\uap}[1]{\uparrow
\rlap{$\scriptstyle{#1}$}} \nc{\defeq}{\stackrel{\rm def}{=}}
\nc{\disp}[1]{\displaystyle{#1}} \nc{\dotcup}{\
\displaystyle{\bigcup^\bullet}\ } \nc{\gzeta}{\bar{\zeta}}
\nc{\hcm}{\ \hat{,}\ } \nc{\hts}{\hat{\otimes}}
\nc{\barot}{{\otimes}} \nc{\free}[1]{\bar{#1}}
\nc{\uni}[1]{\tilde{#1}} \nc{\hcirc}{\hat{\circ}} \nc{\lleft}{[}
\nc{\lright}{]} \nc{\lc}{\lfloor} \nc{\rc}{\rfloor}
\nc{\curlyl}{\left \{ \begin{array}{c} {} \\ {} \end{array}
    \right .  \!\!\!\!\!\!\!}
\nc{\curlyr}{ \!\!\!\!\!\!\!
    \left . \begin{array}{c} {} \\ {} \end{array}
    \right \} }
\nc{\longmid}{\left | \begin{array}{c} {} \\ {} \end{array}
    \right . \!\!\!\!\!\!\!}
\nc{\onetree}{\bullet} \nc{\ora}[1]{\stackrel{#1}{\rar}}
\nc{\ola}[1]{\stackrel{#1}{\la}}
\nc{\ot}{\otimes} \nc{\mot}{{{\boxtimes\,}}}
\nc{\otm}{\overline{\boxtimes}} \nc{\sprod}{\bullet}
\nc{\scs}[1]{\scriptstyle{#1}} \nc{\mrm}[1]{{\rm #1}}
\nc{\margin}[1]{\marginpar{\rm #1}}   
\nc{\dirlim}{\displaystyle{\lim_{\longrightarrow}}\,}
\nc{\invlim}{\displaystyle{\lim_{\longleftarrow}}\,}
\nc{\mvp}{\vspace{0.3cm}} \nc{\tk}{^{(k)}} \nc{\tp}{^\prime}
\nc{\ttp}{^{\prime\prime}} \nc{\svp}{\vspace{2cm}}
\nc{\vp}{\vspace{8cm}} \nc{\proofbegin}{\noindent{\bf Proof: }}
\nc{\proofend}{$\blacksquare$ \vspace{0.3cm}}
\nc{\modg}[1]{\!<\!\!{#1}\!\!>}
\nc{\intg}[1]{F_C(#1)} \nc{\lmodg}{\!
<\!\!} \nc{\rmodg}{\!\!>\!}
\nc{\cpi}{\widehat{\Pi}}
\nc{\sha}{{\mbox{\cyr X}}}  
\nc{\shap}{{\mbox{\cyrs X}}} 
\nc{\shpr}{\diamond}    
\nc{\shp}{\ast} \nc{\shplus}{\shpr^+}
\nc{\shprc}{\shpr_c}    
\nc{\msh}{\ast} \nc{\zprod}{m_0} \nc{\oprod}{m_1}
\nc{\vep}{\varepsilon} \nc{\labs}{\mid\!} \nc{\rabs}{\!\mid}

\nc{\mmbox}[1]{\mbox{\ #1\ }} \nc{\fp}{\mrm{FP}}
\nc{\rchar}{\mrm{char}} \nc{\End}{\mrm{End}} \nc{\Fil}{\mrm{Fil}}
\nc{\Mor}{Mor\xspace} \nc{\gmzvs}{gMZV\xspace}
\nc{\gmzv}{gMZV\xspace} \nc{\mzv}{MZV\xspace}
\nc{\mzvs}{MZVs\xspace} \nc{\Hom}{\mrm{Hom}} \nc{\id}{\mrm{id}}
\nc{\im}{\mrm{im}} \nc{\incl}{\mrm{incl}}
\nc{\mchar}{\rm char} \nc{\nz}{\rm NZ} \nc{\supp}{\mathrm Supp}

\nc{\Alg}{\mathbf{Alg}} \nc{\Bax}{\mathbf{Bax}} \nc{\bff}{\mathbf f}
\nc{\bfk}{{\bf k}} \nc{\bfone}{{\bf 1}} \nc{\bfx}{\mathbf x}
\nc{\bfy}{\mathbf y}
\nc{\base}[1]{\bfone^{\otimes ({#1}+1)}} 
\nc{\Cat}{\mathbf{Cat}}

\nc{\detail}{\marginpar{\bf More detail}
    \noindent{\bf Need more detail!}
    \svp}
\nc{\Int}{\mathbf{Int}} \nc{\Mon}{\mathbf{Mon}}
\nc{\rbtm}{{shuffle }} \nc{\rbto}{{Rota-Baxter }}
\nc{\remarks}{\noindent{\bf Remarks: }} \nc{\Rings}{\mathbf{Rings}}
\nc{\Sets}{\mathbf{Sets}} \nc{\wtot}{\widetilde{\odot}}
\nc{\wast}{\widetilde{\ast}} \nc{\bodot}{\bar{\odot}}
\nc{\bast}{\bar{\ast}} \nc{\hodot}[1]{\odot^{#1}}
\nc{\hast}[1]{\ast^{#1}} \nc{\mal}{\mathcal{O}}
\nc{\tet}{\tilde{\ast}} \nc{\teot}{\tilde{\odot}}
\nc{\oex}{\overline{x}} \nc{\oey}{\overline{y}}
\nc{\oez}{\overline{z}} \nc{\oef}{\overline{f}}
\nc{\oea}{\overline{a}} \nc{\oeb}{\overline{b}}
\nc{\weast}[1]{\widetilde{\ast}^{#1}}
\nc{\weodot}[1]{\widetilde{\odot}^{#1}} \nc{\hstar}[1]{\star^{#1}}
\nc{\lae}{\langle} \nc{\rae}{\rangle}
\nc{\lf}{\lfloor}\nc{\rf}{\rfloor}

\nc{\cala}{{\mathcal A}} \nc{\calb}{{\mathcal B}}
\nc{\calc}{{\mathcal C}}
\nc{\cald}{{\mathcal D}} \nc{\cale}{{\mathcal E}}
\nc{\calf}{{\mathcal F}} \nc{\calg}{{\mathcal G}}
\nc{\calh}{{\mathcal H}} \nc{\cali}{{\mathcal I}}
\nc{\call}{{\mathcal L}} \nc{\calm}{{\mathcal M}}
\nc{\caln}{{\mathcal N}} \nc{\calo}{{\mathcal O}}
\nc{\calp}{{\mathcal P}} \nc{\calr}{{\mathcal R}}
\nc{\cals}{{\mathcal S}} \nc{\calt}{{\mathcal T}}
\nc{\calw}{{\mathcal W}} \nc{\calk}{{\mathcal K}}
\nc{\calx}{{\mathcal X}} \nc{\CA}{\mathcal{A}}

\nc{\fraka}{{\mathfrak a}} \nc{\frakA}{{\mathfrak A}}
\nc{\frakb}{{\mathfrak b}} \nc{\frakB}{{\mathfrak B}}
\nc{\frakc}{{\mathfrak c}}
\nc{\frakD}{{\mathfrak D}} \nc{\frakg}{{\mathfrak g}}
\nc{\frakH}{{\mathfrak H}} \nc{\frakL}{{\mathfrak L}}
\nc{\frakM}{{\mathfrak M}} \nc{\bfrakM}{\overline{\frakM}}
\nc{\frakm}{{\mathfrak m}} \nc{\frakP}{{\mathfrak P}}
\nc{\frakN}{{\mathfrak N}} \nc{\frakp}{{\mathfrak p}}
\nc{\frakS}{{\mathfrak S}}

\font\cyr=wncyr10 \font\cyrs=wncyr7

\begin{document}

\title[Totally compatible dialgebras]{Totally compatible associative and Lie dialgebras, tridendriform algebras and PostLie algebras}
%
%
\author{Yong Zhang}
\address{Department of Mathematics, Zhejiang University, Hangzhou 310027, China}
\email{tangmeng@zju.edu.cn}
\author{Chengming Bai}
\address{Chern Institute of Mathematics \& LPMC, Nankai University, Tianjin 300071, China}
         \email{baicm@nankai.edu.cn}
\author{Li Guo}
\address{Department of Mathematics and Computer Science, Rutgers University, Newark, NJ 07102}
\email{liguo@newark.rutgers.edu}


\date{\today}

\begin{abstract}
This paper studies the concepts of a totally compatible dialgebra and a totally compatible Lie dialgebra, defined to be a vector space with two binary operations that satisfy individual and mixed associativity conditions and Lie algebra conditions respectively. We show that totally compatible dialgebras are closely related to bimodule algebras and semi-homomorphisms. More significantly, Rota-Baxter operators on totally compatible dialgebras provide a uniform framework to generalize known results that Rota-Baxter related operators give tridendriform algebras. Free totally compatible dialgebras are constructed. We also show that a Rota-Baxter operator on a totally compatible Lie dialgebra gives rise to a PostLie algebra, generalizing the fact that a Rota-Baxter operator on a Lie algebra gives rise to a PostLie algebra.
\end{abstract}

\maketitle

\tableofcontents

\setcounter{section}{0}


\section{Introduction}
In recent years, there have been quite much interests on (linearly) compatible products from both mathematics and physics. For example two Lie brackets $[\ ,\ ]_1$ and $[\ ,\ ]_2$ on a vector space are called compatible if all linear combination of the two brackets are still Lie brackets. Such structures are studied in~\mcite{GS1,GS2,GS3,OS3} in the contexts of the classical Yang-Baxter equation and principal chiral field, loop algebras over Lie algebras and elliptic theta functions. The corresponding operad is obtained in~\mcite{DK}. Compatible associative products are studied in~\mcite{OS1,OS2,OS4} in connection with Cartan matrices of affine Dynkin diagrams, integrable matrix equations, infinitesimal bialgebras and quiver representations. In this case, the corresponding operad and free objects are obtained in~\mcite{D}. More general compatible products are defined in~\mcite{St}.

In this paper we consider the Koszul dual of compatible associative dialgebra and compatible Lie dialgebras. These dual structures have been
introduced in~\mcite{DK,St}, called totally compatible associative
and Lie dialgebras respectively. Our motivation in studying such structures comes from the
facts~\mcite{Ag,EF} that Rota-Baxter algebras give rise to
dendriform algebras and tridendriform algebras, and the facts that
Rota-Baxter Lie algebras give the structures of pre-Lie and PostLie algebras. This connection has been studied by several subsequent
papers such as ~\mcite{BGN1,BGN4,CM,EG1,Uc} in the associative
algebra case and~\mcite{BGN2,EF,GS,LHB} in the Lie
algebra case. It has been found that some other structures, such as
TD-algebras~\mcite{Le}, also give tridendriform algebras. See Section~\mref{ss:rbtc} for details. Given the importance of the concepts of Rota-Baxter algebras and tridendriform algebras, it is
interesting to ask what kind of Rota-Baxter related structures can give rise to a tridendriform algebra. We show that totally compatible algebra is a more general structure on which a Rota-Baxter operator gives a
tridendriform algebra, allowing us to give a uniform approach that combine known results such as those mentioned above. We also show that a Lie algebraic analogue holds.

In Section~\mref{sec:2alg}, we begin with the concept of a totally compatible dialgebra and its relationship with $A$-bimodule
$\bfk$-algebras and semi-homomorphisms. We then investigate in Section~\mref{ss:rbtc} Rota-Baxter operators on totally compatible dialgebras and establish their close relationship with tridendriform algebras. Free totally
compatible dialgebras are constructed in Section~\mref{sec:free} and
are used to give examples of Rota-Baxter totally compatible
dialgebras.
The concepts of various compatible Lie dialgebras and Rota-Baxter
totally compatible Lie dialgebras are introduced in Section~\mref{sec:rbcld}.
Their relationship with Rota-Baxter compatible dialgebras and PostLie
algebras is established. This relationship refines the already
established relationship~\mcite{BGN2} of Rota-Baxter Lie algebras and PostLie
algebras.

\section{Totally compatible dialgebras and their Rota-Baxter operators}
\mlabel{sec:2alg}

In this section, we consider totally compatible
dialgebras and their Rota-Baxter operators. The relations of Rota-Baxter totally compatible dialgebras with tridendriform algebras, $A$-bimodule $\bfk$-algebras and semi-homomorphisms are established.

\subsection{Totally compatible dialgebras}

Let $\bfk$ be a commutative unitary ring. The tensor product over {\bf k} is denoted by
$\otimes_{{\bf k}}$ or simply by $\otimes$ if it causes no confusion.

\begin{defn}
{\rm
\begin{enumerate}
\item
A {\bf totally compatible (associative) dialgebra (TCDA)} is a $\bf k$-module $R$ with two binary
operations:
$$\opa, \opb: R\otimes R\longrightarrow R,$$
 satisfying the TCDA axioms:
\begin{enumerate}
\item
$\opa$ and $\opb$ are associative.
\mlabel{it:cda1}
\item
\begin{equation}
(a\opa b)\opb c=a\opa(b\opb c)=(a\opb b)\opa c=a\opb(b\opa c),
\quad \forall a,b,c \in R. \mlabel{eq:2ca}
\end{equation}
\mlabel{it:cda2}
\end{enumerate}
\mlabel{it:cda}
\item
Let $(R,\opa,\opb)$ and $(R{'},\opa{'},\opb{'})$ be two
totally compatible dialgebras. A linear map $f:R\rightarrow R{'}$ is a {\bf
homomorphism of totally compatible dialgebras} if $f$ is a $\bfk$-module homomorphism and, for all $a$, $b\in R$,
\begin{equation}
f(a\opa b)=f(a)\opa'f(b)\quad\text{and}\quad
f(a\opb b)=f(a)\opb'f(b).
\end{equation}
\item
A {\bf totally compatible disemigroup} is a set $S$ with two binary operations $\opa,\opb$ on $S$ that satisfy the TCDA axioms in
Items~(\mref{it:cda}).
\item
The concept of a homomorphism of totally compatible disemigroups is defined in the same way.
\end{enumerate}
}\mlabel{def:cda}
\end{defn}

The operad of totally compatible dialgebras is denoted by ${}^2As$ in~\mcite{St}. It is the Koszul dual of the operad $As^2$ of compatible dialgebras, defined to be $\bfk$-modules $V$ with binary operations $\opa$ and $\opb$ that are associative and satisfy the relation
$$ (x{\opb}y){\opa} z+(x{\opa} y) {\opb}z
=x{\opa}(y{\opb}z) +x {\opb}(y{\opa}z), \quad \forall x,y,z\in V.$$

A compatible dialgebra is also called an algebra with two compatible associative products~\mcite{D,St}. The term compatible comes from the fact that, for a $\bfk$-module $V$ with two associative products $\opa$ and $\opb$, $(V,\opa,\opb)$ is a compatible dialgebra if and only if linear combinations of $\opa$ and $\opb$ are still associative.

It is easy to see that if $(R,\opa,\opb)$ is a totally compatible dialgebra, then, for any $r,s\in \bfk$, the triple $(R,r\opa, s\opb)$ is again a totally compatible dialgebra.
It is also easy to check that the tensor product of two totally compatible dialgebras is
naturally a totally compatible dialgebra. Similarly, let $\mathcal{M}_n(R)$
be the $\bfk$-module of $n\times n$-matrices $\alpha:=(\alpha_{ij})$ with entries $\alpha_{ij}, 1\leq i,j\leq n,$ in a totally compatible
dialgebra $(R,\opa,\opb)$. For $\alpha=(\alpha_{ij}), \beta=(\beta_{ij}) \in \calm_n(R)$, define
$$\alpha\tilde{\opa} \beta:=((\alpha\tilde{\opa}\beta)_{ij}), \quad \alpha\tilde{\opb}\beta:=((\alpha\tilde{\opb}\beta)_{ij}),$$
where
$$(\alpha\tilde{\opa}\beta)_{ij}=\sum_{k=1}^n \alpha_{ik}\opa
\beta_{kj}\,\,\text{and}\,\, (\alpha\tilde{\opb}\beta)_{ij}=\sum_{k=1}^n
\alpha_{ik}\opb \beta_{kj}.
$$
Then $(\mathcal{M}_n(A),\tilde{\opa},\tilde{\opb})$ is a totally compatible dialgebra.

\subsection{Totally compatible dialgebras, $A$-bimodule $\bfk$-algebras and semi-homomorphisms}
We now study the relationship between totally compatible dialgebras and $A$-bimodule $\bfk$-algebra introduced in~\mcite{BGN1} and semi-homomorphisms.

\subsubsection{$A$-bimodule $\bfk$-algebras}
\begin{defn}
{\rm Let $(A, \opa)$ be a {\bf k}-algebra with multiplication
$\opa$.
Let $(R, \ope)$ be a $\bf k$-algebra with multiplication $\ope$. Let
$\ell,r: A \rightarrow \End_{k}(R)$ be two linear maps. We call $(R,\opb,
\ell, r)$ or simply $R$ an {\bf $A$-bimodule $\bf k$-algebra} if $(R,\ell,
r)$ is an $A$-bimodule that is compatible with the multiplication
$\ope$ on $R$ in the sense that the following equations hold.

\begin{equation}
\ell(x \opa y)v=\ell(x)(\ell(y)v), \,\,
vr(x\opa y)=(vr(x))r(y), \,\,
(\ell(x)v)r(y)= \ell(x)(vr(y)),
\mlabel{eq:bim1}
\end{equation}
\begin{equation}
\ell(x)(v\ope
w)=(\ell(x)v)\ope w, \,\,
(v\ope
w)r(x)=v\ope(wr(x)),\,\, \,\,
(vr(x))\ope w=v\ope(\ell(x)w),
\mlabel{eq:bim2}
\end{equation}
for all $x, y \in A, v, w \in R.$
}\mlabel{def:bim}
\end{defn}

\begin{prop} {\rm (\cite{Bai,BGN1})}Let $(A, \opa)$ be a {\bf k}-algebra. Then $(R,\ope, \ell,r)$ is an $A$-bimodule $\bfk$-algebra if and
only if the direct sum $A\oplus R$ of
$\bfk$-modules is turned into a $\bfk$-algebra (the semidirect
sum) by defining multiplication in $A\oplus R$ by
\begin{equation}
(x_1,v_1)\star (x_2,v_2)=(x_1\opa x_2,
\ell(x_1)v_2+v_1r(x_2)+v_1\ope v_2),\;\;\forall x_1,x_2\in A,
v_1,v_2\in R.
\notag 
\end{equation}
\mlabel{co:twoalg}
\end{prop}
We denote this algebra by $A\ltimes_{\ell,r}R$ or simply $A\ltimes R$.

\begin{prop}
Let $(A,\opa,\opb)$ be a totally compatible dialgebra. Define left and right actions by
$$L_\opa(x): A\rightarrow A, L_\opa(x)(y):=x\opa y, $$
$$R_\opa(x): A\rightarrow A, R_\opa(x)(y):=y\opa x,
\,\,\forall x,y \in A.$$
Then the quadruple $(A,\opb,L_\opa,R_\opa)$ is an $A$-bimodule $\bf k$-algebra.
Conversely, if $(A,\opa)$ is a $\bfk$-algebra and $(A,\opb,L_\opa,R_\opa)$ is an $A$-bimodule $\bfk$-algebra
with the left and right actions $L_\opa$ and $R_\opa$ defined above, then the triple $(A,\opa,\opb)$ is
a totally compatible dialgebra. \mlabel{prop:bim}
\end{prop}

\begin{proof}
By the definition of $A$-bimodule $\bf k$-algebra, it suffices to verify the equations in Definition~\mref{def:bim}. It is well-known that, for the associative $\bfk$-algebra $(A,\opa)$, $(A,L_\opa,R_\opa)$ is an $A$-bimodule. Hence the equations in Eq.~(\mref{eq:bim1}) hold. For the equations in Eq.~(\mref{eq:bim2}), we check that

\begin{eqnarray*}
L_\opa(x)(y\opb z)&=&x\opa(y\opb z)=(x\opa
y)\opb z=(L_\opa(x)y)\opb z,
\\
(x\opb y)R_\opa(z)&=&(x\opb y)\opa z=x\opb(y\opa
z)=x\opb(yR_\opa(z)),
\\
(xR_\opa(y))\opb z&=&(x\opa y)\opb z=x\opb(y\opa
z)=x\opb(L_\opa(y)z).
\end{eqnarray*}

Conversely, if $A$ is an $A$-bimodule $\bf k$-algebra, then we have
\begin{eqnarray*}
x\opa(y\opb z)&=&L_\opa(x)(y\opb z)=(L_\opa(x)y)\opb z=(x\opa y)\opb z,\\
(x\opb y)\opa z&=&(x\opb y)R_\opa(z)=x\opb(yR_\opa(z))=x\opb(y\opa
z),\\
(x\opa y)\opb z&=&(xR_\opa(y))\opb z=x\opb(L_\opa(y)z)=x\opb(y\opa z).
\end{eqnarray*}
Since $\opa$ and $\opb$ are already associative, $(A,\opa,\opb)$ is a totally compatible dialgebra.
\end{proof}

\begin{coro} With the conditions as above, the following conditions are equivalent:
\begin{enumerate}
\item $(A,\opa,\opb)$ is a totally compatible dialgebra.
\mlabel{it:cb1}
\item $(A,\opb,L_\opa,R_\opa)$ is an $A$-bimodule $\bf k$-algebra.
\mlabel{it:cb2}
\item There is a $\bfk$-algebra structure $A\ltimes_{L_\opa,R_\opa}A$ on the $\bfk$-module
$A\oplus A$ defined by
\begin{equation} (x,y)\star (z,w)=(x\opa z, x\opa w+y\opa z+y\opb w),\;\;\forall x,y,z,w\in A.
\mlabel{eq:semi}
\end{equation}
\mlabel{it:cb3}
\end{enumerate}
\mlabel{co:semi}
\end{coro}

\begin{proof}
The equivalence between Item~(\mref{it:cb1}) and Item~(\mref{it:cb2}) is just Proposition~\mref{prop:bim}. The equivalence between Item~(\mref{it:cb2}) and
Item~(\mref{it:cb3}) follows from Proposition~\mref{co:twoalg}.
\end{proof}

\subsubsection{Semi-homomorphisms}

\begin{defn}{\rm Let $(A,\cdot)$ be a $\bfk$-algebra. A linear transformation $f:A\rightarrow A$
is called {\bf semi-homomorphism} of $A$ if $f$ satisfies
\begin{equation}
f(x\cdot y)=x\cdot f(y)=f(x)\cdot y,\forall x,y\in A.
\mlabel{eq:shom}
\end{equation}
}
\end{defn}
The set of all semi-homomorphisms is called the centroid of $A$.

\begin{prop}
Let $(A,\cdot)$ be a $\bfk$-algebra and let $f,g$ be commuting semi-homomorphisms on $A$: $fg=gf$.
Define
\begin{equation}
x\opa y:=f(x)\cdot y(=f(x\cdot y)=x\cdot f(y))\;\;\text{and}\;\;
x\opb y:=g(x)\cdot y(=g(x\cdot y)=x\cdot g(y))\;\;\forall\, x,y\in A.
\mlabel{eq:con}\end{equation}
Then $(A,\opa,\opb)$ is a totally compatible dialgebra.
In particular, for any semi-homomorphism $f$ on $A$, define
\begin{equation}
x\opa y:=x\cdot y, \quad x\opb y:=f(x)\cdot y,\quad \forall\, x,y\in A.
\mlabel{eq:con2}\end{equation}
Then $(A,\opa,\opb)$ is a totally compatible dialgebra.
\mlabel{pp:con}
\end{prop}

\begin{proof}
It is straightforward to check that both $(A,\opa)$ and $(A,\opb)$ are $\bfk$-algebras. Let $x, y, z$ be in $A$. Then we have
{\allowdisplaybreaks
\begin{eqnarray*}
&&(x\opa y)\opb z=g(f(x)\cdot y)\cdot z=f(x)\cdot g(y)\cdot z=x\opa (y\opb z);\\
&&(x\opb y)\opa z=f(g(x)\cdot y)\cdot z=g(x)\cdot f(y)\cdot z=x\opb (y\opa z);\\
&&x\opa(y\opb z)=f(x)\cdot g(y)\cdot z=x\cdot fg(y)\cdot z=x\cdot gf(y)\cdot z
=g(x)\cdot f(y)\cdot z=x\opb(y\opa z).
\end{eqnarray*}
}
Therefore $(A,\opa,\opb)$ is a totally compatible dialgebra.

To prove the last statement, we just note that the identity map $\id$ and $f$ are commuting semi-homomorphisms on $A$.
\end{proof}

\begin{exam}
{\rm Let $(A,\opa)$ be a $\bfk$-algebra.
\begin{enumerate}
\item The identity map ${\rm Id}$ on $(A,\opa)$ is obviously a semi-homomorphism. The corresponding totally compatible dialgebra is $(A,\opa,\opa)$.
\item Let $w$ be in the center of $A$ and define $f(x)=x\opa w$ for all $x\in A$. Then $f$ is a semi-homomorphism. Thus by Proposition~\mref{pp:con}, for
$$x\opb y:=x\opa w\opa y(=w\opa x\opa y=x\opa y\opa w), \;\;\forall x,y\in A, $$
the triple $(A,\opa,\opb)$ is a totally compatible dialgebra. This gives another way to obtain the totally compatible dialgebra in Lemma~\mref{lem:ctd}.(\mref{it:rtcd2}). See Remark~\mref{rk:rhom}.
\mlabel{it:chom}
\end{enumerate}}
\mlabel{ex:shom}
\end{exam}

\section{Rota-Baxter totally compatible dialgebras and tridendriform algebras}
\mlabel{ss:rbtc}
In this section we study the close relationship between Rota-Baxter totally compatible dialgebra and tridendriform algebras.
We begin with recalling the basis concepts and results.

\begin{defn}
{\rm
\begin{enumerate}
\item {\rm (\mcite{Ba,Gub,GK,Ro1})}
Let $\lambda\in\bfk$ be given. A linear operator $P$ on a $\bfk$-algebra $R$ is called a {\bf Rota-Baxter operator of weight $\lambda$} if
\begin{equation}
P(x)P(y)=P(xP(y))+P(P(x)y)+P(\lambda xy), \quad \forall x,y\in R,
\mlabel{eq:rbo}
\end{equation}
Then the pair $(R,P)$ is called a {\bf Rota-Baxter algebra of weight $\lambda$}.
\item
{\bf (\mcite{LR1})} A {\bf
tridendriform algebra} is a quadruple $(T,\prec,\succ,\cdot)$
consisting of a $\bfk$-module $T$ and three bilinear products
$\prec$, $\succ$ and $\cdot$ such that \allowdisplaybreaks{
\begin{eqnarray}
&&(x\prec y)\prec z=x\prec (y\star z),\
(x\succ y)\prec z=x\succ (y\prec z),\notag \\
&&(x\star y)\succ z=x\succ (y\succ z),\
(x\succ y)\cdot z=x\succ (y\cdot z),\
\mlabel{eq:tri}\\
&&(x\prec y)\cdot z=x\cdot (y\succ z),\
(x\cdot y)\prec z=x\cdot (y\prec z),\
(x\cdot y)\cdot z=x\cdot (y\cdot z) \notag
\end{eqnarray}
} for all $x,y,z\in T$. Here $\star=\prec+\succ+\cdot.$
\end{enumerate}
}\mlabel{def:dtalg}
\end{defn}

It has been established~\mcite{Ag,EF} that if $P$ is a Rota-Baxter operator of weight $\lambda$ on a $\bfk$-algebra $R$,
then
$$ x\prec_P y:=xP(y),\  x\succ_P y:=P(x)y,\  x\cdot_P y:= \lambda xy, \quad \forall x, y\in R,$$
define a tridendriform algebra structure on $R$. It has been found later that some other Rota-Baxter algebra type structures also give tridendriform algebras. For example, a {\bf TD operator} $P$ on a unitary $\bfk$-algebra $R$ is a linear operator $P:R\to R$ such that
$$P(x)P(y)=P(xP(y)) + P(P(x)y)-P(xP(1)y), \forall x,y\in R.$$
By~\cite[Proposition~2.5]{Le}, a TD operator $T$ also gives a dendriform trialgebra by
$$ x\prec_P y:=xP(y),\  x\succ_P y:=P(x)y,\  x\cdot_P y:= -xP(1)y, \quad \forall x, y\in R.$$
As another example, for a given $w\in R$, define a {\bf Rota-Baxter operator of weight $w$} to be a linear operator $P:R\to R$ such that
$$ P(x)P(y)=P(xP(y))+P(P(x)y) + P(xwy), \quad \forall x, y\in R.$$
Suppose $w$ is in the center of $R$. Then
$$ x\prec_P y:=xP(y), x\succ_P y:=P(x)y, x\cdot_P y:= xwy, \quad \forall x, y\in R,$$
define a tridendriform algebra structure on $R$.

In each of the above cases, the products $xP(1)y$ (resp. $xwy$) are derived from the default product of $R$. Our motivation is to consider a new product on $R$ and study its ``compatibility" conditions with the default product on $R$ in order to still obtain a tridendriform algebra from a Rota-Baxter operator. As we will see in Theorem~\mref{thm:rtca}, such a structure is exactly the totally compatible dialgebra.

We now consider variations of a totally compatible dialgebra that allow a Rota-Baxter operator.
\begin{defn}{\rm
\begin{enumerate}
\item
A {\bf restricted Rota-Baxter totally compatible dialgebra} is a quadruple $(R,\opa,\opb,P)$ where $R$ is a $\bfk$-module, $\opa,\opb$ are associative multiplications on $R$, $P:R\to R$ is a linear operator, satisfying the following compatibility conditions.
\begin{equation}
(P(x)\opa y)\opb z = P(x)\opa (y\opb z), \
(x\opa P(y))\opb z = x\opb (P(y)\opa z), \
(x\opb y) \opa P(z)=x \opb (y \opa P(z)),
\mlabel{eq:rda}
\end{equation}
\begin{equation}
P(x)\opa P(y)=P(x\opa P(y))+P(P(x)\opa y)+P(x\opb y), \quad \forall x, y, z\in R.
\mlabel{eq:rrbe}
\end{equation}
\item
Let $(R,\opa,\opb)$ be a totally compatible dialgebra. A linear operator $P:R
\to R$ is called a {\bf Rota-Baxter
operator} if
$$P(x)\opa P(y)=P(x\opa P(y))+P(P(x)\opa y)+P(x\opb y), \quad \forall x,y\in R.$$
If the equation holds, then $(R,\opa,\opb,P)$ is called a {\bf Rota-Baxter totally compatible dialgebra}.
\end{enumerate}
\mlabel{def:rb2a} }
\end{defn}
\begin{remark}
{\rm
\begin{enumerate}
\item
Note that we did not assign a weight to the Rota-Baxter operator in Definition~\mref{def:rb2a}. This is because the effect of a weight can be achieved by a variation of the product $\opb$: instead of considering $\lambda x\opb y$, we could consider $x \opb' y$ with $x\opb' y:=\lambda x\opb y$. Such an instance can be found in Proposition~\mref{pp:fcdf}.
\item
It follows from the definitions that a Rota-Baxter totally compatible dialgebra is a restricted Rota-Baxter totally compatible dialgebra.
If $P$ is surjective, then the two concepts agree with each other.
\end{enumerate}
}
\end{remark}

The following lemma shows that the concept of restricted Rota-Baxter totally compatible dialgebra gives a suitable context to combine known cases of Rota-Baxter type operators that give tridendriform algebras. See also Corollary~\mref{co:ctd}.

\begin{lemma}
\begin{enumerate}
\item
Let $R$ be a $\bfk$-algebra, let $w$ be in $R$ and let $P:R\to R$ be a linear operator. Suppose $wP(x)=P(x)w$ for all $x\in R$ and
$$P(x)P(y)=P(xP(y)+P(x)y+xwy), \quad \forall x,y\in R.$$
Define
\begin{equation}
 x\opa y:=xy,\quad x\opb y:=xwy, \quad \forall x,y\in R.
\mlabel{eq:wprod}
\end{equation}
Then the quadruple $(R,\opa,\opb,P)$ is a restricted Rota-Baxter totally compatible
dialgebra.
\mlabel{it:rtcd1}
\item
Let $R$ be a $\bfk$-algebra and let $w\in R$ be in the center of $R$. Let $P:R\to R$ be a Rota-Baxter operator of weight $w$. Then the quadruple $(R,\opa,\opb,P)$ with $\opa,\opb$ defined in Eq.~(\mref{eq:wprod})
is a restricted Rota-Baxter totally compatible dialgebra. \mlabel{it:rtcd2}
\item
Let $R$ be a unitary $\bfk$-algebra and let $P:R\to R$ be a TD operator. Then the quadruple $(R,\opa,\opb,P)$ where
$$ x\opa y:= xy, \ x\opb y:= xP(1)y, \quad \forall x,y\in R,$$
is a restricted Rota-Baxter totally compatible dialgebra.
\mlabel{it:rtcd3}
\end{enumerate}
\mlabel{lem:ctd}
\end{lemma}

\begin{remark}
{\rm The triple $(R,\opa,\opb)$ in Item~(\mref{it:rtcd2}) is in fact a totally compatible dialgebra, as observed in Example~\mref{ex:shom}.(\mref{it:chom}), from semi-homomorphisms. }
\mlabel{rk:rhom}
\end{remark}

\begin{proof} (\mref{it:rtcd1})
The multiplications $\opa$ and $\opb$ are associative by their definitions. To verify Eq.~(\mref{eq:rda}) we compute
\begin{eqnarray*}
(P(x)\opa y)\opb z&=& (P(x)y)wz = P(x)(ywz)=P(x)\opa (y\opb z),\\
(x\opa P(y))\opb z &=& xP(y)wz=xwP(y)z=x\opb (P(y)\opa z), \\
(x\opb y)\opa P(z)&=& (xwy)P(z)=xw(yP(z))=x\opb (y\opa P(z)),
\quad \forall x,y,z\in R.
\end{eqnarray*}
Eq.~(\mref{eq:rrbe}) is automatic. Therefore $(R,\opa,\opb,P)$ is a restricted Rota-Baxter totally compatible dialgebra.

\smallskip

\noindent
(\mref{it:rtcd2}) This is a direct consequence of Item~(\mref{it:rtcd1}).
\smallskip

\noindent
(\mref{it:rtcd3}) For a TD operator $P$, we have~\mcite{Le}
$$
P(x)P(1)=P(xP(1)+P(x)-xP(1))=P^2(x)=P(P(x)+P(1)x-P(1)x))=P(1)P(x), \quad \forall x\in R.
$$
Thus we just need to take $w=P(1)$ in Item~(\mref{it:rtcd1}).
\end{proof}

The following result shows the close relationship between restricted Rota-Baxter totally compatible dialgebras and tridendriform algebras.

\begin{theorem}
Let $(R,\opa,\opb,P)$ be a {\bf k}-module with two associative multiplications $\opa, \opb$ and a linear map $P:R\to R$ such that
$$P(x)\opa P(y)=P(x\opa P(y)+P(x)\opa y+x\opb y), \quad \forall x, y\in R.$$
Define
\begin{equation}
x\prec_P y:=x\opa P(y),\, x\succ_P y:=P(x)\opa y, \,x\cdot_P y:= x\opb y.
\mlabel{eq:rbtd}
\end{equation}
Then $(R,\prec_P, \succ_P, \cdot_P)$ is a tridendriform algebra if and only if $(R,\opa,\opb,P)$ is a restricted Rota-Baxter totally compatible dialgebra.
\mlabel{thm:rtca}
\end{theorem}

\begin{proof}
($\Leftarrow$). We just need to verify the seven axioms for the operations $\prec_P,\succ_P$ and $\cdot_P$ in the definition of a tridendriform algebra in Definition~\mref{def:dtalg}. Denote
$\star_P=\prec_P+\succ_P+\cdot_P.$
From the associativity of $\opa$, we obtain
$$(x\star_P y)\succ z=(P(x)\opa P(y))\opa z=P(x)\opa(P(y)\opa z)=x\succ_P(y\succ_P z),$$
$$(x\prec_P y)\prec_P z=(x\opa P(y))\opa P(z)=x\opa(P(y)\opa P(z))=x\prec_P(y\star_P z)$$
and
$$(x\succ_P y)\prec_P z=(P(x)\opa y)\opa P(z)=P(x)\opa(y\opa P(z))=x\succ_P(y\prec_P z).$$
From Eq.~(\mref{eq:rda}), we have
$$(x\cdot_P y)\prec_P z=(x\opb y)\opa P(z)=x\opb(y\opa P(z))=x\cdot_P(y\prec_P z),$$
$$(x\succ_P y)\cdot_P z=(P(x)\opa y)\opb z=P(x)\opa(y\opb z)=x\succ_P(y\cdot_P z)$$
and
$$(x\prec_P y)\cdot_P z=(x\opa P(y))\opb z=x\opb(P(y)\opa z)=x\cdot_P(y\succ_P z).$$
Finally from the associativity of $\opb$, we obtain
$$(x\cdot_P y)\cdot_P z=x\cdot_P(y\cdot_P z).$$
Thus we have verified all the axioms for a tridendriform algebra.
\smallskip

\noindent
($\Rightarrow$).
Let $(R,\opa,\opb,P)$ be as given in the theorem and suppose $(R,\prec_P,\succ_P,\cdot_P)$ defined by Eq.~(\mref{eq:rbtd}) is a tridendriform algebra. Then we also have
$$P(x)\opa P(y)=P(x\opa
P(y)+P(x)\opa y+x\opb y)=P(x\star y).$$
Then the axioms of the dendriform algebra $(R,\prec_P,\succ_P,\cdot_P)$ imply
\begin{enumerate}
\item
$(P(x)\opa y)\opb z=P(x)\opa(y\opb z)$,
\item
$(x\opa P(y))\opb z=x\opb(P(y)\opa z)$,
\item
$(x\opb y)\opa P(z)=x\opb(y\opa P(z))$,
\end{enumerate}
giving Eq.~(\mref{eq:rda}).
\end{proof}

The following corollary follows directly from Lemma~\mref{lem:ctd} and Theorem~\mref{thm:rtca}.
\begin{coro}
\begin{enumerate}
\item
If $(R,\opa,\opb,P)$ is a Rota-Baxter totally compatible dialgebra, then
$(R,\prec_P,\succ_P,\cdot_P)$ defined in Eq.~(\mref{eq:rbtd}) is a tridendriform algebra.
\mlabel{it:cddt1}
\item
If $P$ is surjective, then the converse of Item~(\mref{it:cddt1}) holds.
\mlabel{it:cddt2}
\item
Let $R$ be a $\bfk$-algebra and let $w\in R$ be such that $wP(x)=P(x)w$ for all $x\in R$.
Let $(R,P)$ be a Rota-Baxter algebra of weight $w$. Then
$$ x\prec_P y:=xP(y), x\succ_P y:=P(x)y, x\cdot_P y:= xwy, \quad \forall x, y\in R,$$
define a tridendriform algebra structure on $R$.
\item
{\rm (\mcite{Le})} Let $(R,P)$ be a TD algebra. Then
$$ x\prec_P y:=xP(y), x\succ_P y:=P(x)y, x\cdot_P y:= xP(1)y, \quad \forall x, y\in R,$$
define a tridendriform algebra structure on $R$.
\end{enumerate}
\mlabel{co:ctd}
\end{coro}

We end this section by showing that a Rota-Baxter operator on a totally compatible dialgebra is equivalent to a Rota-Baxter operator on a larger $\bfk$-algebra.

\begin{theorem}
Let $(R,\opa,\opb)$ be a totally compatible dialgebra. Let $(R\ltimes_{L_\opa,R_\opa}R,\star)$ be the $\bf k$-algebra in Corollary~\mref{co:semi} with the product $\star$ defined by Eq.~(\mref{eq:semi}). Then a linear operator $P:R\to R$ is a Rota-Baxter operator on the totally compatible dialgebra $(R,\opa,\opb)$, in the sense of Definition~\mref{def:rb2a},
if and only if
$$\hat {P}: R\ltimes_{L_\opa,R_\opa} R\rightarrow R\ltimes_{L_\opa, R_\opa}R,
\quad \hat {P}(x,y)=(-x +P(y),0),\;\;\forall x, y\in R,$$
is a Rota-Baxter operator of weight 1 on the $\bfk$-algebra $R\ltimes_{L_\opa,R_\opa} R$.
\end{theorem}

\begin{proof}
Let $x,y,z,w\in R$. Then
{\allowdisplaybreaks
\begin{eqnarray*}
\hat P(x,y) \star \hat P(z,w)&=&(-x+P(y),0)\star (-z+P(w),0)\\
&=&(x\opa z-x\opa P(w)-P(y)\opa z+P(y)\opa P(w),0);\\
\hat P(\hat P(x,y)\star (z,w))&=&\hat P((-x+P(y),0)\star (z,w))=\hat P(-x\opa z+P(y)\opa z, -x\opa w+P(y)\opa w)\\
&=&(x\opa z-P(y)\opa z-P(x\opa w)+P(P(y)\opa w),0);\\
\hat P ((x,y)\star \hat P(z,w))&=& \hat P((x,y)\star (-z+P(w),0))=\hat P(-x\opa z+x\opa P(w), -y\opa z+y\opa P(w))\\
&=& (x\opa z-x\opa P(w)-P(y\opa z)+P(y\opa P(w)),0);\\
\hat P((x,y)\star (z,w))&=& \hat P(x\opa z,x\opa w+y\opa z+y\opb w)\\
&=&(-x\opa z+P(x\opa w)+P(y\opa z)+P(y\opb w),0).
\end{eqnarray*}
}
Therefore $P(y)\opa P(w)=P(y\opa P(w)+P(y)\opa w+y\opb w)$ if and only if
$$\hat P(x,y) \star \hat P(z,w)=\hat P(\hat P(x,y)\star (z,w))+\hat P ((x,y)\star \hat P(z,w))+\hat P((x,y)\star (z,w)).$$
That is, $P:R
\to R$ is a Rota-Baxter
operator on $(R,\opa,\opb)$ if and only if $\hat {P}$ is a Rota-Baxter operator of weight 1 on $R\ltimes_{L_\opa,R_\opa} R$.
\end{proof}

\section{Free totally compatible dialgebras}
\mlabel{sec:free}

In this section, we construct free totally compatible dialgebras and use them to provide examples of Rota-Baxter totally compatible dialgebras.

Let $X$ be a set. Let
$$S(X)=\left\{x_1\cdots x_n\,\Big|\,x_i\in X, 1\leq i\leq n, n\geq 1\right\}$$
be the free semigroup generated by $X$ and let $M(X)$ be the free monoid generated by $X$. Then
the semigroup algebra $\bfk\langle X\rangle^0:=\bfk\, S(X)$ is the noncommutative nonunitary polynomial algebra generated by $X$ and the semigroup algebra $\bfk\langle X\rangle:=\bfk\, M(X)$ is the noncommutative unitary polynomial algebra generated by $X$. Consider the $\bfk$-module underlying the
tensor product
\begin{equation}
F(X):=\bfk\langle X\rangle^0\ot \bfk\langle X\rangle \cong \bfk (S(X)\times M(X)),
\mlabel{eq:ftcda}
\end{equation}
with $S(X)\times M(X)$ as a canonical basis. As such, $(u,v)\in S(X)\times M(X)$, with $u\in S(X)$ and $v\in M(X)$, corresponds to $u\ot v$ in $F(X)$. Thus we will use $u\ot v$ to denote $(u,v)$.

With these notations, we define two binary operations
$\bar{\opa}$ and $\bar{\opb}$ on $S(X)\times M(X)$ as follows. Consider two elements

\begin{equation}
\fraka=(x_1 \cdots x_m)\otimes (x_{m+1} \cdots x_{m+n}), \quad
\frakb=(x_{m+n+1}\cdots x_{m+n+k})\otimes (x_{m+n+k+1} \cdots x_{m+n+k+l})
\mlabel{eq:freeab}
\end{equation}
in $S(X)\times M(X)$, where $m,k\geq 1$ and $n,\ell\geq 0$ with the convention that $x_{m+1}\cdots x_{m+n}=1$ if $n=0$ and $x_{m+n+k+1}\cdots x_{m+n+k+\ell}=1$ if $\ell=0$. Here $1$ is the identity in $\bfk$.
Define
\begin{eqnarray}
\fraka\bast \frakb&=&(x_1 \cdots x_{m+k})\otimes (x_{m+k+1} \cdots x_{m+n+k+l}),
\mlabel{eq:opa1} \\
\fraka\bodot \frakb&=&(x_1 \cdots x_{m+k-1})\otimes (x_{m+k} \cdots x_{m+n+k+l}).
\mlabel{eq:opb1}
\end{eqnarray}
In particularly, when $n=\ell=0$, we define
\begin{eqnarray}
((x_1 \cdots x_m)\otimes 1)\free{\opa} ((x_{m+1} \cdots
x_{m+k})\otimes 1)&=&(x_1 \cdots x_{m+k})\otimes 1, \mlabel{eq:opa2}
\\
((x_1 \cdots x_m)\otimes 1) \free{\opb}((x_{m+1} \cdots
x_{m+k}))\otimes 1)&=&(x_1 \cdots x_{m+k-1})\otimes x_{m+k}.
\mlabel{eq:opb2}
\end{eqnarray}
These two binary operations expand to binary operations on $F(X)$ by $\bfk$-bilinearity that we still denote by $\free{\opa}$ and $\free{\opb}$.

\begin{theorem} Let $X$ be a set.
\begin{enumerate}
\item
The set $S(X)\times M(X)$, with the multiplications $\free{\opa}$ and $\free{\opb}$ defined in Eqs.~(\mref{eq:opa1}) -- (\mref{eq:opb2}), and the embedding
\begin{equation}
i_X:X\to S(X)\times M(X), \quad x\mapsto x\ot 1, x\in X,
\mlabel{eq:cdemb}
\end{equation}
is the free totally compatible disemigroup on $X$. \mlabel{it:frees}
\item
The $\bfk$-module $F(X)=\bfk\langle X\rangle^0\ot \bfk\langle X\rangle$, with the multiplications $\free{\opa}$ and $\free{\opb}$ defined in Eqs.~(\mref{eq:opa1}) -- (\mref{eq:opb2}), and the embedding
\begin{equation}
i_X:X\to F(X), \quad x\mapsto x\ot 1, x\in X,
\mlabel{eq:cdemba}
\end{equation}
is the free totally compatible dialgebra on $X$. \mlabel{it:freea}
\end{enumerate}
\mlabel{thm:free2al}
\end{theorem}
\begin{proof}
(\mref{it:frees}).
We first check that the triple $(S(X)\times M(X),\free{\opa},\free{\opb})$ is a totally compatible disemigroup. For $\fraka,\frakb$ in Eq.~(\mref{eq:freeab}) and
$$\frakc=(x_{m+n+k+\ell+1} \cdots x_{m+n+k+\ell+p})\ot(x_{m+n+k+\ell+p+1} \cdots
x_{m+n+k+\ell+p+q})$$
with $p\geq 1$ and $q\geq 0$. We have
$$ (\fraka\free{\opa} \frakb)\free{\opa} \frakc=
(x_1\ot \cdots x_{m+k+p})\ot (x_{m+k+p+1}\cdots x_{m+k+p+n+\ell+q})=\fraka\free{\opa}(\frakb\free{\opa}\frakc).$$
Hence $\free{\opa}$ is associative. Similarly,
$$ (\fraka\free{\opb} \frakb)\free{\opb} \frakc=
(x_1\ot \cdots x_{m+k+p-2})\ot (x_{m+k+p-1}\cdots x_{m+k+p+n+\ell+q})=\fraka\free{\opb}(\frakb\free{\opb} \frakc),$$
showing that $\free{\opb}$ is associative.

On the other hand, for these choices of $a, b, c$, each term in Eq.~(\mref{eq:2ca}) equals to
$$(x_1 \cdots x_{m+k+p-1})\otimes (x_{m+k+p} \cdots x_{m+n+k+\ell+p+q}),$$
proving Eq.~(\mref{eq:2ca}).

We next show that the totally compatible disemigroup $S(X)\times M(X)$ is free on
$X$ by checking that $(S(X)\times M(X),\bar{\opa},\bar{\opb})$
satisfies the universal property of a free totally compatible disemigroup
over $X$: Let $(S,\opa,\opb)$ be a totally compatible
disemigroup with operations $\opa$ and $\opb$. Let
$f:X\to S$ be a set map. Then there exists a unique
homomorphism
$$\overline {f}:S(X)\times M(X)\longrightarrow S$$ of totally compatible disemigroups such that
$f=\overline{f} \opb i_X$ for the $i_X$ defined in Eq.~(\mref{eq:cdemb}).

Let a set map $f:X\to S$ be given.
Define
\begin{eqnarray}
\free{f}:&& S(X)\ot M(X) \to S, \notag \\
(x_1\cdots x_m)\ot (x_{m+1}\cdots x_{m+n})
&&\mapsto (f(x_1)\opa\cdots \opa f(x_m))\opb (f(x_{m+1})\opb \cdots
\opb f(x_{m+n})),\mlabel{eq:freefdef}
\\
(x_1\cdots x_m)\ot 1 &&\mapsto f(x_1)\opa \cdots \opa f(x_m). \notag
\end{eqnarray}
We will show that $\free{f}$ is the unique extension of $f$ that is a homomorphism of totally compatible disemigroups.

We first show that $\free{f}$ is indeed a homomorphism of totally compatible disemigroups. For any
$$\fraka=(x_1 \cdots x_m)\otimes (x_{m+1} \cdots x_{m+n}), \quad
\frakb=(x_{m+n+1}\cdots x_{m+n+k})\otimes (x_{m+n+k+1} \cdots x_{m+n+k+l})$$
in $S(X)\ot M(X)$ as expressed in Eq.~(\mref{eq:freeab}), we have
{\allowdisplaybreaks
\begin{eqnarray*}
&&\free{f}(\fraka)\opa \free{f}(\frakb)\\
&=&
\big((f(x_1)\opa \cdots\opa f(x_m))\opb (f(x_{m+1})\opb \cdots \opb f(x_{m+n}))\big)
\\
&&\opa
\big((f(x_{m+n+1})\opa\cdots \opa f(x_{m+n+k}))\opb (f(x_{m+n+k+1})\opb \cdots \opb f(x_{m+n+k+l}))\big)
\quad \text{(by Eq.~(\mref{eq:freefdef}))} \\
&=&
(f(x_1)\opa \cdots \opa f(x_{m+k})) \opb (f(x_{m+k+1})\opb \cdots \opb f(x_{m+k+n+\ell}))
\quad \text{(by Eq.~(\mref{eq:2ca}))} \\
&=&
\free{f}((x_1\opa \cdots \opa x_{m+k}) \ot (x_{m+k+1} \cdots x_{m+k+n+\ell}))
\quad \text{(by Eq.~(\mref{eq:freefdef}))} \\
&=&
\free{f}(\fraka \free{\opa} \frakb). \quad \text{(by Eq.~(\mref{eq:opa1}))}
\end{eqnarray*}
}
Similarly, we have
{\allowdisplaybreaks
\begin{eqnarray*}
&&\free{f}(\fraka)\opb \free{f}(\frakb)\\
&=&
\big((f(x_1)\opa \cdots\opa f(x_m))\opb (f(x_{m+1})\opb \cdots \opb f(x_{m+n}))\big)
\\
&&\opb
\big((f(x_{m+n+1})\opa\cdots \opa f(x_{m+n+k}))\opb (f(x_{m+n+k+1})\opb \cdots \opb f(x_{m+n+k+l}))\big)
\quad \text{(by Eq.~(\mref{eq:freefdef}))} \\
&=&
(f(x_1)\opa \cdots \opa f(x_{m+k-1})) \opb (f(x_{m+k})\opb \cdots \opb f(x_{m+k+n+\ell}))
\quad \text{(by Eq.~(\mref{eq:2ca}))} \\
&=&
\free{f}((x_1\opa \cdots \opa x_{m+k-1}) \ot (x_{m+k} \cdots x_{m+k+n+\ell}))
\quad \text{(by Eq.~(\mref{eq:freefdef}))} \\
&=&
\free{f}(\fraka \free{\opa} \frakb). \quad \text{(by Eq.~(\mref{eq:opb1}))}
\end{eqnarray*}
}

Thus it remains to show that any homomorphism of totally compatible disemigroups from $S(X)\ot M(X)$ that extends $f$ is uniquely determined by Eq.~(\mref{eq:freefdef}). Suppose
$$\free{f}':S(X)\ot M(X) \to S$$
is a homomorphism of totally compatible disemigroups that extends $f$. We will prove $\free{f}'=\free{f}.$

Note that the set
$$ \calx:=S(X)\times M(X):=\{(x_1\cdots x_m)\ot (x_{m+1}\cdots x_{m+n})\,|\, x_i\in X, 1\leq i\leq m+n, m\geq 1, n\geq 0\} $$
is a disjoint union
\begin{equation}
\calx = \calx_1\sqcup \calx_2 \sqcup \calx_3,
\mlabel{eq:xcup}
\end{equation}
where
\begin{eqnarray}
\calx_1:&=&\{(x_1\cdots x_m)\ot (x_{m+1}\cdots x_{m+n})\in \calx\,|\, m\geq 1, n= 0\}\\
&=&\{(x_1\cdots x_m)\ot 1\,|\, x_i\in X, 1\leq i\leq m, m\geq 1\},
\notag\\
\calx_2:&=&\{(x_1\cdots x_m)\ot (x_{m+1}\cdots x_{m+n})\in \calx\,|\, m= 1, n\geq 1\}\\
&=&\{x_1\ot (x_2\cdots x_{n+1})\,|\, x_i\in X, 1\leq i\leq n+1, n\geq 1\},
\notag\\
\calx_3:&=&\{(x_1\cdots x_m)\ot (x_{m+1}\cdots x_{m+n})\in \calx\,|\, m\geq 2, n\geq 1\}.
\end{eqnarray}
Thus in order to prove $\free{f}'=\free{f}$, we just need to show that $\free{f}'$ agrees with $\free{f}$ on the three subsets $\calx_i, 1\leq i\leq 3,$ of $\calx$.

\smallskip

\noindent
{\bf Case 1. $\free{f}'=\free{f}$ on $\calx_1$.}
Note that, by Eq.~(\mref{eq:opa2}), $\calx_1=S(X)\ot 1$ is the free semigroup on $X$ with respect to the product $\free{\opa}$. Since $\free{f}'$ is in particular a semigroup homomorphism with respect to $\free{\opa}$, $\free{f}'$ must agree with the unique extension of $f$ to a semigroup homomorphism
$$(S(X)\ot 1,\free{\opa})\to (R,\opa).$$
Thus we must have
\begin{eqnarray}
&&\free{f}'((x_1\cdots x_m)\ot 1)
\notag\\
&=& \free{f}((x_1\ot 1)\free{\opa}\cdots \free{\opa} (x_m\ot 1))\mlabel{eq:freef1}\\
&=& \free{f}(x_1\ot 1) \opa \cdots \opa\free{f}(x_m\ot 1)\notag\\
&=& f(x_1)\opa \cdots \opa f(x_m). \notag
\end{eqnarray}
This agrees with the definition of $\free{f}$ in Eq.~(\mref{eq:freefdef}).
\smallskip

\noindent
{\bf Case 2. $\free{f}'$ agrees with $\free{f}$ on $\calx_2$.}
We first prove a lemma.
\begin{lemma}
For any $n\geq 1$, we have
\begin{equation}
x_1\ot (x_2\cdots x_{n+1})=(x_1\ot 1)\free{\opb} \cdots \free{\opb}(x_{n+1}\ot 1).
\mlabel{eq:1g1}
\end{equation}
\mlabel{lem:1g1}
\end{lemma}
\begin{proof}
We prove Eq.~(\mref{eq:1g1}) by induction on $n\geq 1$. When $n=1$, this follows from Eq.~(\mref{eq:opb2}). Assume that Eq.~(\mref{eq:1g1}) has been proved for $n=k\geq 1$ and consider the case when $n=k+1$. So we have $x_1\ot (x_2\cdots x_{k+2})$. By Eq.~(\mref{eq:opb2}) and the induction hypothesis, we have
\begin{eqnarray*}
x_1\ot (x_2\cdots x_{k+1})&=& (x_1\ot 1) \free{\opb} (x_2\ot (x_3\cdots x_{k+2}))\\
&=& (x_1\ot 1)\free{\opb} ((x_2\ot 1)\free{\opb} \cdots \free{\opb}(x_{k+1}\ot 1)),
\end{eqnarray*}
as needed.
\end{proof}

Thus, in order to get a totally compatible disemigroup homomorphism, the restriction of $\free{f}'$ on $\calx_2$ must satisfy
\begin{eqnarray}
\free{f}'(x_1\ot (x_2\cdots x_{n+1}))&=&
\free{f}'((x_1\ot 1)\free{\opb} \cdots \free{\opb}(x_{n+1}\ot 1))\notag \\
&=& \free{f}'(x_1\ot 1) \opb \cdots \opb \free{f}'(x_{n+1}\ot 1) \mlabel{eq:freef2}\\
&=& f(x_1)\opb \cdots \opb f(x_{n+1}).
\quad \text{(by Eq.~(\mref{eq:freef1}))} \notag
\end{eqnarray}
This again agrees with $\free{f}$ in Eq.~(\mref{eq:freefdef}).
\smallskip

\noindent
{\bf Case 3. $\free{f}'$ agrees with $\free{f}$ on $\calx_3$.}
By Eq.~(\mref{eq:opa1}), we have
$$(x_1\cdots x_m)\ot (x_{m+1}\cdots x_{m+n}) =
((x_1\cdots x_{m-1})\ot 1)\free{\opa} (x_m\ot (x_{m+1}\cdots x_{m+n})).$$
Thus the $\free{f}'$ in Eq.~(\mref{eq:freef2}) can be uniquely defined on $$\free{f}': \calx_3 \to S$$
by
\begin{eqnarray}
&&\free{f}'((x_1\cdots x_m)\ot (x_{m+1}\cdots x_{m+n})) \notag \\
&=&
\free{f}'(((x_1\cdots x_{m-1})\ot 1)\free{\opa} (x_m\ot (x_{m+1}\cdots x_{m+n}))) \notag
\\&=&
\free{f}'(((x_1\ot 1)\free{\opa}\cdots \free{\opa}(x_{m-1}\ot 1))\free{\opa} ((x_m\ot 1)\free{\opb}\cdots \free{\opb}(x_{m+n}\ot 1)))\mlabel{eq:freef3}\\
&=&(\free{f}'(x_1\ot 1)\opa\cdots \opa\free{f}'(x_{m-1}\ot 1))\opa (\free{f}'(x_m\ot 1)\opb\cdots \opb \free{f}'(x_{m+n}\ot 1))\notag\\
&=&
(f(x_1)\opa\cdots \opa f(x_{m-1}))\opa(f(x_m)\opb \cdots \opb f(x_{m+n})).\notag
\end{eqnarray}
This again agrees with $\free{f}$. Therefore $\free{f}'=\free{f}$ and the uniqueness of $\free{f}$ is proved.
\smallskip

\noindent
(\mref{it:freea}). Let $\mathit{TCDA}$, $\mathit{TCDS}$ and $\mathit{Sets}$ be the category of totally compatible dialgebras, the category of totally compatible semigroups and the category of sets, respectively. Then the forgetful functor from $\mathit{TCDA}$ to $\mathit{Sets}$ is the composition of the forgetful functor from $\mathit{TCDA}$ to $\mathit{TCDS}$ followed by the forgetful functor from $\mathit{TCDS}$ to $\mathit{Sets}$. Thus the corresponding adjoint (free) functor from $\mathit{Sets}$ to $\mathit{TCDA}$ is the composition of the free functor from $\mathit{Sets}$ to $\mathit{TCDS}$ followed by the free functor from $\mathit{TCDS}$ to $\mathit{TCDA}$. The former is given by Item~(\mref{it:frees}) and the latter is given by taking the disemigroup algebra:
$$ \bfk (S(X)\times M(X)) \cong (\bfk S(X)) \ot (\bfk M(X)) = \bfk\langle X\rangle^0 \ot \bfk\langle X\rangle,$$
with the operations $\free{\opa}$ and $\free{\opb}$ extended by bilinearity.
\end{proof}

As an application of Theorem~\mref{thm:free2al}, we show that free totally compatible algebras provide examples of Rota-Baxter totaly compatible dialgebras as well as totally compatible dialgebras.
\begin{prop}
Let $(F(X),\free{\opa},\free{\opb})$ be the free totally compatible dialgebra on $X$ in Theorem~\mref{thm:free2al}. Then $(F(X),\free{\opa},-\,\free{\opb})$ has a
Rota-Baxter operator P given by
 $$P((x_1\cdots x_m)\ot(x_{m+1} \cdots x_{m+n}))=(x_1\cdots x_{m+n})\ot 1,$$
$$P((x_1 \cdots x_m)\ot 1)=(x_1 \cdots x_m)\ot 1,\,
\quad \forall x_i\in X, 1\leq i\leq m+n.$$
\mlabel{pp:fcdf}
\end{prop}
\vspace{-.3cm}
\begin{proof}
Let $x=(x_1 \cdots x_m)\ot(x_{m+1} \cdots x_{m+n})$ and
$y=(x_{m+n+1} \cdots x_{m+n+k})\ot(x_{m+n+k+1} \cdots x_{m+n+k+\ell})$.
Then from the definitions of $\free{\opa}, \free{\opb}$ and $P$, we directly check that each of $P(x)\free{\opa}P(y)$,
$P(x\bar{\opa} P(y))$, $P(P(x)\bar{\opa} y)$ and $P(x\bar{\opb} y)$ equals to $(x_1 \cdots x_{m+n+k+\ell})\ot 1$.
Hence,
$$P(x)\bar{\opa} P(y)=P(x\bar{\opa} P(y))+P(P(x)\bar{\opa}
y)+P(x(-\bar{\opb}) y).$$
Therefore, $(F(X),\bar{\opa},\bar{\opb},P)$ is Rota-Baxter totally compatible dialgebra.
\end{proof}

\section{Totally compatible Lie dialgebras, Rota-Baxter operators and PostLie algebras}
\mlabel{sec:rbcld}

In this section we study variations of compatible Lie dialgebra and Rota-Baxter operators on them. We generalize the relationship between associative algebras and Lie algebras to the relationship between totally compatible dialgebras and totally compatible Lie dialgebras. We also generalize the relationship between Rota-Baxter Lie algebras and PostLie algebras to the relationship between Rota-Baxter totally compatible Lie dialgebras and PostLie algebras.

\subsection{Compatible Lie dialgebras}
The origin of compatible Lie dialgebras is the following definition.
\begin{defn}\mcite{GS1,GS2,Ku,OS3}
{\rm Let  $(V,[\ ,\ ]_1)$ and $(V,[\ ,\ ]_2)$ be two Lie algebras. They are called {\bf compatible} if
for any $\alpha,\beta\in \bf k$, the following product
\begin{equation}
[x, y]=\alpha [x, y]_1+\beta [x, y]_2,\forall x,y\in V,
\end{equation}
defines a Lie algebra.}
\end{defn}

\begin{prop}\mcite{Ku,GS1,GS2,OS3}. Let $(V,\opa)$ and $(V,\opb)$ be two Lie algebras. Then the following conditions are equivalent:
\begin{enumerate}
\item  $(V,[\ ,\ ]_1)$ and $(V,[\ ,\ ]_2)$ are compatible.
\item The following equation holds:
\begin{equation}
[[x,y]_1,z]_2+[[z,x]_1,y]_2+[[y,z]_1,x]_2+[[x,y]_2,z]_1+[[z,x]_2,y]_1+[[y,z]_2,x]_1=0
,\;\;\forall x,y,z\in V. \mlabel{eq:Lie-com}\end{equation}
\end{enumerate}
\end{prop}

Following~\mcite{St}, we give the following definitions.
\begin{defn}
{\rm
Consider a triple $(V,[\ ,\ ]_1,[\ ,\ ]_2)$ where $[\ ,\ ]_1$ and $[\ ,\ ]_2$ are Lie brackets.
\begin{enumerate}
\item
The triple is called a {\bf compatible Lie dialgebra} if Eq.~(\mref{eq:Lie-com}) holds.
\item
The triple is called a {\bf totally compatible Lie dialgebra} if
\begin{eqnarray}
&[ [x,y]_1,z]_2= [[x,y]_2,z]_1, & \mlabel{eq:tcl1}\\
&[[x,y]_1,z]_2+[[z,x]_1,y]_2+[[y,z]_1,x]_2= 0, \quad \forall x, y, z\in V. &
\mlabel{eq:tcl2}
\end{eqnarray}
\end{enumerate}
}
\end{defn}
A compatible Lie dialgebra is a $Lie^2$-algebra for the operad $Lie^2$ in~\cite[Proposition~1.15]{St} and a totally compatible Lie dialgebra is a ${}^2Lie$-algebra for the operad ${}^2Lie$ in~\cite[Definition~1.13]{St}.

\begin{remark}
{\rm If $(V,[\ ,\ ]_1,[\ ,\ ]_2)$ is a totally compatible Lie dialgebra, then for any choice of $\{i_1,i_2\}=\{j_1,j_2\}=\{k_1,k_2\}=\{1,2\}$, we have
\begin{equation}
[[x,y]_{i_1},z]_{i_2}+[[z,x]_{j_1},y]_{j_2} +[[y,z]_{k_1},x]_{k_2}=0.
\mlabel{eq:lieg}
\end{equation}
}\end{remark}

\begin{prop}
Let $(A,\opa,\opb)$ be a totally compatible dialgebra. Define
\begin{equation}
[x,y]_1:=x\opa y-y\opa x,\,[x,y]_2:=x\opb y-y\opb x,\forall x,y\in A.\mlabel{eq:cond}
\end{equation}
Then $(A,[\ ,\ ]_1,[\ ,\ ]_2)$ is a totally compatible Lie dialgebra.
\mlabel{pp:alcda}
\end{prop}
\begin{proof}
Since $\opa$ and $\opb$ are associative, the brackets $[\ ,\ ]_1$ and $[\ ,\ ]_2$ are Lie brackets. Note that for any $x,y,z\in A$, we have
\begin{eqnarray*} &&[[x,y]_1,z]_2=(x\opa y)\opb z-(y\opa x)\opb z-z\opb(x\opa y)
+z\opb(y\opa x),\\
&&[[x,y]_2,z]_1=(x\opb y)\opa z-(y\opb x)\opa z-z\opa(x\opb y)
+z\opa(y\opb x).\end{eqnarray*}
Hence Eq.~(\mref{eq:tcl1}) holds.
Eq.~(\mref{eq:tcl2}) is verified by a similar calculation.
\end{proof}

\subsection{Rota-Baxter totally compatible Lie dialgebras and PostLie algebras}

We recall the concept of a PostLie algebra~\mcite{Va}.
\begin{defn}
{\rm
A {\bf PostLie algebra} is a
$\bfk$-module $L$ with two bilinear operations $\circ$
and $[\ ,\ ]$ that satisfy the relations:
\begin{equation}
[x,y]=-[y,x],\mlabel{eq:polie1}
\end{equation}
\begin{equation}
[[x,y],z]+[[z,x],y]+[[y,z],x]=0,\mlabel{eq:polie2}
\end{equation}
\begin{equation}
x\circ(y\circ z)-y\circ(x\circ z)+(y\circ x)\circ z-(x\circ y)\circ
z+[y,x]\circ z=0,\mlabel{eq:polie3}
\end{equation}
\begin{equation}
z\circ[x,y]-[z\circ x,y]-[x,z\circ y]=0, \forall x,y\in L.
\mlabel{eq:polie4}
\end{equation}
}\end{defn}

\begin{theorem}
Consider a triple $(V,[\ ,\ ]_1,[\ ,\ ]_2,P)$, where
\begin{enumerate}
\item $[\ ,\ ]_1$ and $[\ ,\ ]_2$ are Lie brackets,
\item
\begin{equation}
[[x,y]_1,z]_2+[[z,x]_2,y]_1+[[y,z]_1,x]_2=0, \forall x,y,z\in V.
\mlabel{eq:lcom3}
\end{equation}
\item
\begin{equation}
[P(x),P(y)]_1=P[P(x),y]_1+P[x,P(y)]_1+P[x,y]_2, \forall x, y\in V.
\mlabel{eq:rbcl}
\end{equation}
\end{enumerate}
Then the operations
 \begin{equation}
 x\circ y:=[P(x),y]_1, \quad [x,y]:=[x,y]_2, \quad \forall x,y\in V,
 \mlabel{eq:clpl}
\end{equation}
define a PostLie algebra $(V,\circ,[\ ,\ ])$.
In particular, if $(V,[\ ,\ ]_1, [\ ,\ ]_2)$ is a totally compatible Lie dialgebra, then the operations in Eq.~(\mref{eq:clpl}) define a PostLie algebra.
\mlabel{thm:cldpl}
\end{theorem}
\begin{proof}
We only need to verify Eqs.~(\mref{eq:polie3}) and (\mref{eq:polie4}).

By Eq.~(\mref{eq:clpl}), the left hand side of Eq.~(\mref{eq:polie3}) is
{\allowdisplaybreaks
\begin{eqnarray*}
&&[P(x),[P(y),z]_1]_1-[P[P(x),y]_1,z]_1-[P(y),[P(x),z]_1]_1 +[P[P(y),x]_1,z]_1+[P[y,x]_2,z]_1 \\
&=&[P(x),[P(y),z]_1]_1 +[\left(P[y,P(x)]_1+P[P(y),x]_1\right),z]_1 -[P(y),[P(x),z]_1]_1+[P[y,x]_2,z]_1 \\
&=&[P(x),[P(y),z]_1]_1 +[\left([P(y),P(x)]_1-P([y,x]_2)\right),z]_1 \\
&&-[P(y),[P(x),z]_1]_1+[P[y,x]_2,z]_1 \quad \text{(by Eq.~(\mref{eq:rbcl}))}\\
&=&[P(x),[P(y),z]_1]_1 +[[P(y),P(x)]_1,z]_1 -[P(y),[P(x),z]_1]_1 \\
&=&[[z,P(y)]_1,P(x)]_1+[[P(y),P(x)]_1,z]_1 +[[P(x),z]_1,P(y)]_1
\end{eqnarray*}
}
which is zero by the Jacobi identity of the Lie bracket $[\ ,\ ]_1$.

On the other hand, the left hand side of Eq.~(\mref{eq:polie4}) is
$$[P(z),[x,y]_2]_1-[[P(z),x]_1,y]_2-[x,[P(z),y]_1]_2\\
= -[[x,y]_2,P(z)]_1-[[P(z),x]_1,y]_2-[[y,P(z)]_1,x]_2
$$
which is zero by Eq.~(\mref{eq:lcom3}) where $x, y, z$ are replaced by $y, P(z), x$ respectively.
\end{proof}

The following commutative diagram of categories has been established in
~\mcite{BGN2,BGN4}.
\begin{equation}
\xymatrix{
{\begin{array}{c}\text{Rota-Baxter algebras}\\ \text{of weight 1\ } (R,\opa,P) \end{array}} \ar[rrr]^{[x,y]:=x\opa y-y\opa x}
\ar[dd]_{\tiny{\begin{array}{l}x\prec y:=x\opa P(y)\\
x\succ y:=P(x)\opa y\\ x \cdot y:=x\opa y\end{array}}} &&&
{\begin{array}{c}\text{Rota-Baxter Lie algebras} \\ \text{of weight 1\ } (R,[\ ,\ ],P) \end{array}}
\ar[dd]_{x\circ y:=[P(x),y]}^{[x,y]:=[x,y]} \\
&&&\\
{\begin{array}{c}\text{Tridendriform algebras\ } \\ (R,\prec,\succ,\cdot)\end{array}}
\ar[rrr]^{x\circ y:=x\succ y-y\prec x}_{[x,y]:=x\cdot y-y\cdot x} &&& {\begin{array}{c}\text{PostLie algebras\ } \\
(R,\circ,[\ ,\ ])\end{array}}
}
\notag \mlabel{di:postlie}
\end{equation}

To summarize, we have the following refinement of the above commutative diagram.
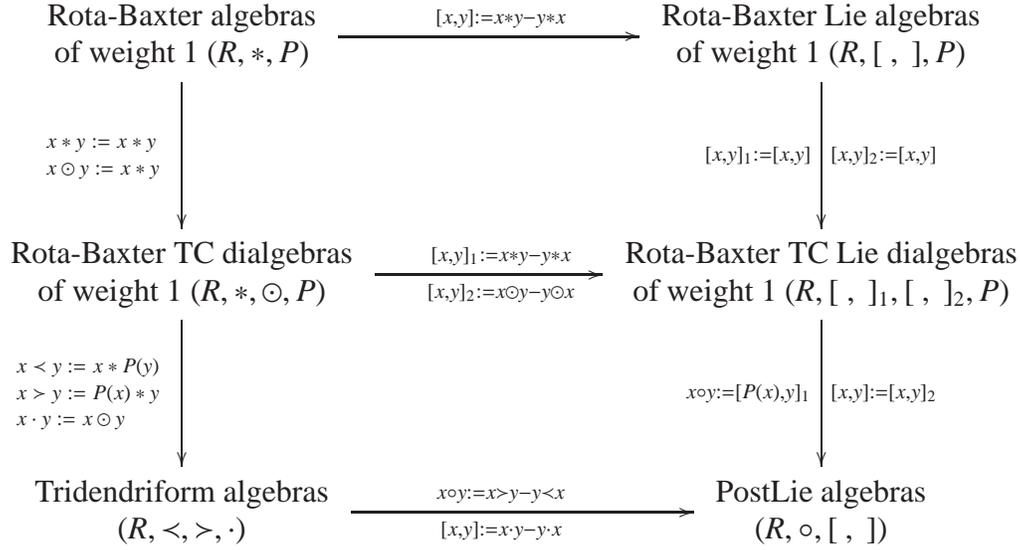
\begin{equation}
\xymatrix{
{\begin{array}{c}\text{Rota-Baxter algebras}\\ \text{of weight 1\ } (R,\opa,P) \end{array}} \ar[rrr]^{[x,y]:=x\opa y-y\opa x}
\ar[dd]_{\tiny{\begin{array}{l}x\opa y:=x\opa y\\
x\opb y:=x\opa y\end{array}}} &&&
{\begin{array}{c}\text{Rota-Baxter Lie algebras} \\ \text{of weight 1\ } (R,[\ ,\ ],P) \end{array}}
\ar[dd]_{[x,y]_1:=[x,y]}^{[x,y]_2:=[x,y]} \\
&&&\\
{\begin{array}{c}\text{Rota-Baxter TC dialgebras}\\ \text{of weight 1\ } (R,\opa,\opb,P) \end{array}} \ar[rrr]^{[x,y]_1:=x\opa y-y\opa x}_{[x,y]_2:=x\opb y-y\opb x}
\ar[dd]_{\tiny{\begin{array}{l}x\prec y:=x\opa P(y)\\
x\succ y:=P(x)\opa y\\ x \cdot y:=x\opb y\end{array}}} &&&
{\begin{array}{c}\text{Rota-Baxter TC Lie dialgebras} \\ \text{of weight 1\ } (R,[\ ,\ ]_1,[\ ,\ ]_2,P) \end{array}}
\ar[dd]_{x\circ y:=[P(x),y]_1}^{[x,y]:=[x,y]_2} \\
&&&\\
{\begin{array}{c} \text{Tridendriform algebras} \\ (R,\prec,\succ,\cdot)\end{array}} \ar[rrr]^{x\circ y:=x\succ y-y\prec x}_{[x,y]:=x\cdot y-y\cdot x} &&& {\begin{array}{c}\text{PostLie algebras}\\ (R,\circ,[\ ,\ ])\end{array}}
}
\notag
\end{equation}
Here the top two vertical maps are inclusions of categories and $TC$ is the abbreviation of totally compatible.

\smallskip

\noindent
{\bf Acknowledgement }
C. Bai would like to thank the support by NSFC (10920161) and SRFDP
(200800550015). L. Guo thanks NSF grant DMS-1001855 for support.

\end{document}